\documentclass[11pt]{article}
\usepackage[utf8]{inputenc}
\usepackage{amsthm}
\usepackage{amsfonts}
\usepackage{amsmath}
\usepackage{colonequals}
\usepackage{epsfig}
\usepackage{url}
\newtheorem{theorem}{Theorem}
\newtheorem{corollary}[theorem]{Corollary}
\newtheorem{lemma}[theorem]{Lemma}
\newtheorem{observation}[theorem]{Observation}
\newcommand{\GG}{{\cal G}}
\newcommand{\PP}{{\cal P}}
\DeclareMathOperator{\tw}{\text{tw}}
\DeclareMathOperator{\pll}{\text{polylog}}

\title{A note on sublinear separators and expansion}
\author{Zden\v{e}k Dvo\v{r}\'ak\thanks{Computer Science Institute, Charles University, Prague, Czech Republic. E-mail: {\tt rakdver@iuuk.mff.cuni.cz}.
Supported in part by ERC Synergy grant DYNASNET no. 810115.}}
\date{}

\begin{document}
\maketitle

\begin{abstract}
For a hereditary class $\GG$ of graphs, let $s_\GG(n)$ be the minimum function such that each $n$-vertex graph in $\GG$ has a balanced separator of order at most $s_\GG(n)$, and let
$\nabla_\GG(r)$ be the minimum function bounding the expansion of $\GG$, in the sense of bounded expansion theory of Nešetřil and Ossona de Mendez.
The results of Plotkin, Rao, and Smith (1994) and Esperet and Raymond (2018) imply that if $s_\GG(n)=\Theta(n^{1-\varepsilon})$ for some $\varepsilon>0$, then
$\nabla_\GG(r)=\Omega(r^{\frac{1}{2\varepsilon}-1}/\pll r)$ and $\nabla_\GG(r)=O(r^{\frac{1}{\varepsilon}-1}\pll r)$.  Answering a question of Esperet and Raymond, we show that neither of the exponents
can be substantially improved.
\end{abstract}

For an $n$-vertex graph $G$, a set $X\subseteq V(G)$ is a \emph{balanced separator} if each component of $G-X$ has at most $2n/3$ vertices.
Let $s(G)$ denote the minimum size of a balanced
separator in $G$, and for a class $\GG$ of graphs, let $s_\GG:\mathbb{N}\to\mathbb{N}$ be defined by
$$s_\GG(n)=\max\{s(G):G\in\GG, |V(G)|\le n\}.$$
Let us remark that this notion is related to the \emph{separation profile} studied for infinite graphs~\cite{sepprof}.
Classes with sublinear separators (i.e., classes $\GG$ with $s_\GG(n)=o(n)$) are of interest from the computational perspective,
as they naturally admit divide-and-conquer style algorithms.  They also turn out to have a number of intriguing structural properties;
the relevant one for this note is the connection to the density of shallow minors.

For a graph $G$ and an integer $r\ge 0$, an \emph{$r$-shallow minor of $G$} is any graph obtained from a subgraph of $G$ by contracting pairwise vertex-disjoint subgraphs, each of radius at most $r$.
The \emph{density} of a graph $H$ is $|E(H)|/|V(H)|$.
We let $\nabla_r(G)$ denote the maximum density of an $r$-shallow minor of $G$.
For a class $\GG$ of graphs, let $\nabla_\GG:\mathbb{N}\to\mathbb{R}\cup\{\infty\}$ be defined by
$$\nabla_\GG(r)=\sup\{\nabla_r(G):G\in\GG\}.$$
If $\nabla_\GG(r)$ is finite for every $r$, we say that the class $\GG$ has \emph{bounded expansion}.
The classes with bounded expansion have a number of common properties and computational applications;
we refer the reader to~\cite{nesbook} for more details.
The first connection between sublinear separators and bounded expansion comes from the work of Plotkin, Rao, and Smith~\cite{plotkin}.
\begin{theorem}[Plotkin, Rao, and Smith~\cite{plotkin}]\label{thm-prs}
For each $n$-vertex graph $G$ ($n\ge 2$) and all integers $l,h\ge 1$, either $G$ has a balanced separator of order
at most $n/l+2h^2l\log_2 n$, or $G$ contains a $(2l\log_2 n)$-shallow minor of $K_h$.
\end{theorem}

As observed in~\cite{espsublin} (and qualitatively in~\cite{grad2,dvorak2016strongly}), this has the following consequence.
\begin{corollary}\label{cor-lower}
Suppose $\GG$ is a class of graphs such that $s_\GG(n)=\Omega(n^{1-\varepsilon})$ for some $\varepsilon>0$.
Then $\nabla_\GG(r)=\Omega(r^{\frac{1}{2\varepsilon}-1}/\pll r)$.
\end{corollary}
\begin{proof}
Consider a sufficiently large integer $r$, and let
$n=\lfloor r^{\frac{1}{\varepsilon}}\log_2^{-2/\varepsilon} r\rfloor$,
$l=\lfloor \tfrac{1}{2}r\log_2^{-1} n\rfloor$, and
$h=\lfloor \tfrac{1}{2}n^{1/2}l^{-1}\log_2^{-1} n\rfloor$.
Note that 
\begin{align*}
2h^2l\log_2 n&<n/l\\
n^\varepsilon&\le r\log_2^{-2} r\le r\\
\log_2 n&\le \tfrac{1}{\varepsilon}\log_2 r\\
2l\log_2 n&\le r\le 3l\log_2 n.
\end{align*}
Consequently, we have
$$\frac{n}{l}+2h^2l\log_2 n\le\frac{2n}{l}\le \frac{6n\log_2 n}{r}\le 6n^{1-\varepsilon}\frac{\log_2 n}{\log_2^2 r}\le \frac{6}{\varepsilon\log_2 r}n^{1-\varepsilon}<s_\GG(n).$$
Let $G\in\GG$ be a graph with at most $n$ vertices and with no balanced separator
of order less than $s_\GG(n)$; then Theorem~\ref{thm-prs} implies $G$ contains an $r$-shallow minor of $K_h$,
implying that
$$\nabla_\GG(r)\ge \nabla_r(G)\ge \frac{|E(K_h)|}{|V(K_h)|}=\Omega(h)=\Omega(r^{\frac{1}{2\varepsilon}-1}/\pll r).$$
\end{proof}

Dvořák and Norin~\cite{dvorak2016strongly} proved that surprisingly, a converse to Corollary~\ref{cor-lower} holds as well.
Subsequently, Esperet and Raymond~\cite{espsublin} gave a simpler argument with a better exponent:
they state their result with $O(r^{\frac{1}{\varepsilon}}\pll r)$ bound, but an analysis of their argument shows that the
exponent can be improved by $1$. We include the short proof for completeness; the proof uses the following result establishing the connection
between separators and treewidth.
\begin{theorem}[Dvořák and Norin~\cite{dnorin}]\label{thm-tw}
Let $G$ be a graph and $k$ an integer.
If $s(H)\le k$ for every induced subgraph $H$ of $G$, then $\tw(G)\le 15k$.
\end{theorem}
We also need a simple observation on shallow minors and subdivisions.
\begin{observation}\label{obs-subdiv}
For every integer $r\ge 0$, if $H$ is an $r$-shallow minor of a graph $G$
and $H$ has maximum degree at most three, then a subgraph of $G$ is isomorphic
to a graph obtained from $H$ by subdividing each edge at most
$4r$ times.
\end{observation}
\begin{proof}
Since $H$ has maximum degree at most three, we can without loss of generality assume
each subgraph $S_v$ contracted to form a vertex $v\in V(H)$ is a subdivision of a star with at most three rays.
Since $S_v$ has radius at most $r$, the distance from the center of $S_v$ to each leaf is at most $2r$.
For each edge $uv\in E(H)$, let $e_{uv}$ be an edge of $G$ with one end in $S_u$ and the other end in $S_v$;
then $S_u+e_{uv}+S_v$ contains a path from the center of $S_u$ to the center of $S_v$ of length at most $4r+1$.
Consequently, the union of the graphs $S_v$ for $v\in V(H)$ and the edges $e_{uv}$ for $uv\in E(H)$
gives a subgraph of $G$ is isomorphic
to a graph obtained from $H$ by subdividing each edge at most $4r$ times.
\end{proof}
For $\alpha>0$, a graph $G$ is an \emph{$\alpha$-expander} if
$|N(S)|\ge \alpha|S|$ holds for every set $S\subseteq V(G)$ of size at most $|V(G)|/2$.

\begin{theorem}[Esperet and Raymond~\cite{espsublin}]\label{thm-upper}
Suppose $\GG$ is a hereditary class of graphs such that $s_\GG(n)=O(n^{1-\varepsilon})$ for some $\varepsilon>0$.
Then $\nabla_\GG(r)=O(r^{\frac{1}{\varepsilon}-1}\pll r)$.
\end{theorem}
\begin{proof}
Let $H$ be an $r$-shallow minor of a graph $G\in\GG$ and let $d$ be the density of $H$.
By the result of Shapira and Sudakov~\cite{shapira2015small}, there exists a subgraph $H_1\subseteq H$
of average degree $\Omega(d)$ such that, letting $n=|V(H_1)|$, the graph $H_1$ is a $(1/\pll n)$-expander.
Consequently, $H_1$ has treewidth $\Omega(n/\pll n)$.  As Chekuri and Chuzhoy~\cite{chekuri2014degree} proved,
$H_1$ has a subcubic subgraph $H_2$ of treewidth $\Omega(n/\pll n)$.  Since $H_2$ is a subcubic $r$-shallow minor
of $G$, Observation~\ref{obs-subdiv} implies that $G$ has a subgraph $G_2$ obtained from $H_2$ by subdividing each edge at most
$4r$ times, and thus $|V(G_2)|=O(r|V(H_2)|)=O(rn)$.  Furthermore, since $H_2$ is a minor of $G_2$, we have
\begin{equation}\label{eq-a}
\tw(G_2)\ge\tw(H_2)=\Omega(n/\pll n).
\end{equation}
On the other hand, since $\GG$ is hereditary, $G_2$ is a spanning subgraph of a graph from $\GG$, 
and thus every subgraph of $G_2$ has a balanced separator of size $O(|V(G_2)|^{1-\varepsilon})=O((rn)^{1-\varepsilon})$.
By Theorem~\ref{thm-tw}, this implies $\tw(G_2)=O((rn)^{1-\varepsilon})$.  Combining this inequality with (\ref{eq-a}),
this gives $n=O(r^{\frac{1}{\varepsilon}-1}\pll n)=O(r^{\frac{1}{\varepsilon}-1}\pll r)$.  Since $H_1$ has $n$ vertices and average degree
$\Omega(d)$, we have $d=O(n)=O(r^{\frac{1}{\varepsilon}-1}\pll r)$.  This holds for every $r$-shallow minor of a graph from $\GG$, and thus
$\nabla_\GG(r)=O(r^{\frac{1}{\varepsilon}-1}\pll r)$.
\end{proof}

For $0<\varepsilon\le 1$,
\begin{itemize}
\item let $b_\varepsilon$ denote the supremum of real numbers $b$ for which every hereditary class $\GG$ of graphs such that $s_\GG(n)=\Theta(n^{1-\varepsilon})$
satisfies $\nabla_\GG(r)=\Omega(r^b)$, and
\item let $B_\varepsilon$ denote the infimum of real numbers $B$ for which every hereditary class $\GG$ of graphs such that $s_\GG(n)=\Theta(n^{1-\varepsilon})$
satisfies $\nabla_\GG(r)=O(r^B)$.
\end{itemize}
Corollary~\ref{cor-lower} and Theorem~\ref{thm-upper} give the following bounds.
\begin{corollary}\label{cor-known}
For $0<\varepsilon\le 1$,
$$\max\Bigl(\frac{1}{2\varepsilon}-1,0\Bigr)\le b_\varepsilon\le B_\varepsilon\le \frac{1}{\varepsilon}-1.$$
\end{corollary}
Esperet and Raymond~\cite{espsublin} asked whether either of these bounds (in particular, in terms of multiplicative constants) can be improved.
They suggest some insight into this question could be obtained by investigating the $d$-dimensional grids.  While the grids ultimately do not give
the best bounds we obtain, their analysis is instructive and we give it (for even $d$) in the following lemma.

Note that $b_{1/2}=0$, matching the lower bound from Corollary~\ref{cor-known}: Indeed, as proved by Lipton and Tarjan~\cite{lt79},
the class $\PP$ of planar graphs satisfies $s_\PP(n)=\Theta(n^{1/2})$,
and on the other hand, every minor of a planar graph is planar, implying $\nabla_\PP(r)\le 3=O(r^0)$.
However, 2-dimensional grids with diagonals give $B_{1/2}=1$, as we will show in greater generality in the next lemma.  Hence,
$b_\varepsilon$ is not always equal to $B_\varepsilon$.
\begin{lemma}\label{lemma-grids}
For every even integer $d$,
$$b_{1/d}\le \frac{d}{2}\le B_{1/d}.$$
\end{lemma}
\begin{proof}
Let $Q_n^d$ denote the graph whose vertices are elements of $\{1,\ldots,n\}^d$ and two distinct vertices are adjacent
if they differ by at most $2$ in each coordinate.  Let $\GG_d$ denote the class consisting of graphs $Q_n^d$ for all $n\in\mathbb{N}$ and their induced subgraphs.
Note that $s_\GG(n)=\Theta(n^{1-1/d})$: Each induced subgraph $H$ of $Q_n^d$ can be represented as an intersection graph of axis-aligned unit cubes in $\mathbb{R}^d$
where each point is contained in at most $3^d$ cubes, and such graphs have balanced separators of order $O(|V(H)|^{1-1/d})$, see e.g.~\cite{teng1991unified}.
Conversely, standard isoperimetric inequalities show that $Q_n^d$ does not have a balanced separator smaller than $\Omega(n^{d-1})=\Omega(|V(Q_n^d)|^{1-1/d})$.
We claim that $\nabla_\GG(r)=\Theta(r^{d/2})$.

Consider any $r$-shallow minor $H$ of $Q_n^d$, and for $v\in V(H)$, let $B_v$ denote the subgraph of $Q_n^d$ of radius at most $r$ contracted to form $v$.
We have $\Delta(Q_n^d)<5^d$, and thus $\deg_H(v)<5^d|V(B_v)|$ holds for every $v\in V(H)$.  Let $v$ be the vertex of $H$ with $|V(B_v)|$ minimum, and let $c$
be a vertex of $Q_n^d$ such that each vertex of $B_v$ is at distance at most $r$ from $c$.  Note that if $uv\in V(H)$, then every vertex of $B_u$ is at
distance at most $3r+1$ from $c$.  Consequently, the pairwise vertex-disjoint subgraphs $B_u$ for $u\in N(v)$ are all contained
in a cube with side of length $12r+4$ centered at $c$, implying
$$(\deg_H(v)+1)|V(B_v)|\le |V(B_v)|+\sum_{u\in N(v)} |V(B_u)|\le (12r+5)^d,$$
and thus $\deg_H(v)<(12r+5)^d/|V(B_v)|$.  Therefore, since $\min(ax,b/x)\le \sqrt{ab}$ for every $a,b,x>0$, we have
$$\deg_H(v)<\min\bigl(5^d|V(B_v)|,(12r+5)^d/|V(B_v)|\bigr)\le (60r+25)^{d/2}.$$
Hence, each $r$-shallow minor of $Q_n^d$ has minimum degree $O(r^{d/2})$, and thus we have $\nabla_\GG(r)=O(r^{d/2})$.

On the other hand, consider the graph $Q_{2r}^d$.  For $x\in \{1,\ldots,2r\}^{d/2}$, let $A_x$ be the subgraph of $Q_{2r}^d$ induced
by vertices $(i_1, \ldots, i_d)$ such that $i_j=x_j$ for $j=1,\ldots, d/2$ and $i_j\in \{1,\ldots,2r\}$ for $j=d/2+1,\ldots, d$,
and let $B_x$ be the subgraph induced by vertices $(i_1, \ldots, i_d)$ such that $i_1\in \{2,4,\ldots,2r\}$, $i_j\in \{1,\ldots,2r\}$ for $j=2,\ldots, d/2$,
and $i_j=x_{j-d/2}$ for $j=d/2+1,\ldots, d$.
Each of these subgraphs has radius at most $r$, for all distinct $x,x'\in \{1,\ldots,2r\}^{d/2}$ we have $V(A_x)\cap V(A_{x'})=\emptyset$ and $V(B_x)\cap V(B_{x'})=\emptyset$,
and for all $x\in \{1,\ldots,2r\}^{d/2}$ such that $x_1$ is odd and $y\in \{1,\ldots,2r\}^{d/2}$, the graphs $A_x$ and $B_y$ are vertex-disjoint and $Q_{2r}^d$
contains an edge with one end $(x,y)\in V(A_x)$ and the other end in $(x+e_1,y)\in V(B_y)$.  Consequently, $K_{(2r)^{d/2}/2,(2r)^{d/2}}$ is an $r$-shallow minor of $Q_{2r}^d$,
implying $\nabla_\GG(r)=\Omega(r^{d/2})$.
\end{proof}

Lemma~\ref{lemma-grids} implies that $b_\varepsilon\le \tfrac{1}{2\varepsilon}$ when $\tfrac{1}{\varepsilon}$ is an even integer, and thus at these points the lower bound from
Corollary~\ref{cor-known} cannot be improved by more than $1$.  Actually, we can prove an even better bound for all values of $\varepsilon>0$.
To this end, let us first establish bounds on the size of balanced separators in certain graph classes.

\begin{lemma}\label{lemma-twtos}
Let $f,t:\mathbb{R}^+\to\mathbb{R}^+$ be non-decreasing functions, and let $p:\mathbb{R}^+\to\mathbb{R}^+$ be the inverse to the function $x\mapsto xt(x)$.
Let $G$ be a graph such that every induced subgraph $H$ of $G$ satisfies $s(H)\le f(|V(H)|)$.  Let $G'$ be a graph obtained from $G$ by subdividing each edge
at least $t(|V(G)|)$ times.  Then $s(H')\le 15f(p(2|V(H')|))+1$ for every induced subgraph $H'$ of $G'$.
\end{lemma}
\begin{proof}
Without loss of generality, we can assume $H'$ is connected, as otherwise it suffices to consider the size of a balanced separator in the largest component of $H'$.
Let $B$ be the set of vertices of $G'$ created by subdividing the edges, and let $A=V(H')\setminus B$ and $a=|A|$.
If $a\le 1$, then $H'$ is a tree, and thus it has balanced separator of size at most $1$.  Hence, assume that $a\ge 2$.
Since $H'$ is connected, we have $|V(H')|\ge (a-1)t(|V(G)|)\ge at(a)/2$, and thus $a\le p(2|V(H')|)$.
Note that $H'$ is obtained from a subgraph $H$ of $G$ with $a$ vertices by subdividing edges and repeatedly adding pendant vertices.
By Theorem~\ref{thm-tw} we have $\tw(H)\le 15f(a)$, and thus
$\tw(H')\le \tw(H)\le 15f(a)\le 15f(p(2|V(H')|))$.  As proved in~\cite{rs2}, every graph
of treewidth at most $c$ has a balanced separator of order at most $c+1$.  Consequently, $H'$ has a balanced separator
of order at most $15f(p(2|V(H')|))+1$.
\end{proof}

For a graph $G$ with $m$ vertices and $0<\varepsilon<1$, let
$G^\varepsilon$ denote the graph obtained from $G$ by subdividing each edge $\lceil m^{\varepsilon/(1-\varepsilon)}\rceil$ times.  For a class of graphs $\GG$, let $\GG^\varepsilon$ denote
the class consisting of all induced subgraphs of the graphs $G^\varepsilon$ for $G\in \GG$.  

\begin{lemma}\label{lemma-subdivsep}
For every class of graphs $\GG$ and every $0<\varepsilon<1$, we have $s_{\GG^{\varepsilon}}(n)=O(n^{1-\varepsilon})$.
If $\GG$ contains all $3$-regular graphs, then $s_{\GG^{\varepsilon}}(n)=\Omega(n^{1-\varepsilon})$.
\end{lemma}
\begin{proof}
Applying Lemma~\ref{lemma-twtos} with $f(n)=n$ and $t(m)=\lceil m^{\varepsilon/(1-\varepsilon)}\rceil$ (so that $p(n)=\Theta(n^{1-\varepsilon})$),
we have $s_{\GG^\varepsilon}(n)=O(n^{1-\varepsilon})$.

Conversely, let $G\in \GG$ be a $3$-regular $\tfrac{3}{20}$-expander with $m=\Theta(n^{1-\varepsilon})$ vertices
(such a graph exists for every sufficiently large even number of vertices~\cite{bolo}).  Note that $|V(G^{\varepsilon})|=\Theta(m\cdot m^{\varepsilon/(1-\varepsilon)})=\Theta(m^{1/(1-\varepsilon)})=\Theta(n)$.
We now argue that $s(G^{\varepsilon})=\Omega(m)$, which implies $s_\GG(n)=\Omega(m)=\Omega(n^{1-\varepsilon})$.

Let $M$ be the set of vertices of $G^{\varepsilon}$ of degree three.
Suppose for a contradiction $X$ is a balanced separator in $G^{\varepsilon}$ of size $o(m)$.  For sufficiently large $n$, this implies $V(G^\varepsilon)$
can be expressed as disjoint union of $X$, $C_1$, and $C_2$, where $C_1$ and $C_2$ are unions of components of $G^\varepsilon-X$ and
$|C_1|,|C_2|\ge |V(G^{\varepsilon})|/4=\Omega(n)$.  Each component of $G^\varepsilon-X$ disjoint from $M$ has two neighbors in $X$, implying the total number
of vertices in such components is at most $\tfrac{3}{2}|X|\lceil m^{\varepsilon/(1-\varepsilon)}\rceil=o(n)$.
Furthermore, a component of $G-X$ containing $k\ge 1$ vertices of $M$ has $O(km^{\varepsilon/(1-\varepsilon)})$ vertices.
Consequently, $|C_1\cap M|,|C_2\cap M|=\Omega(n/m^{\varepsilon/(1-\varepsilon)})=\Omega(m)$.
By symmetry, we can assume $|C_1\cap M|\le m/2$, and since $G$ is a $\tfrac{3}{20}$-expander, we have
$N_G(C_1\cap M)=\Omega(m)$.  However, this implies $|X|=\Omega(m)$, which is a contradiction.
\end{proof}

Applying this lemma with $\GG$ consisting of all $3$-regular graphs, we obtain the following bound.

\begin{lemma}\label{lemma-sdex}
For $0<\varepsilon\le 1$, $b_\varepsilon\le \frac{1}{2\varepsilon}-\frac{1}{2}$.
\end{lemma}
\begin{proof}
We have $b_1=0$ by Corollary~\ref{cor-known}, and thus we can assume $\varepsilon<1$.
Let $\GG_3$ be the class of all $3$-regular graphs.  By Lemma~\ref{lemma-subdivsep}, we have $s_{\GG_3^\varepsilon}(n)=\Theta(n^{1-\varepsilon})$.

Let $G$ be a $3$-regular graph with $m$ vertices, and consider any $r$-shallow minor $F$ of $G^\varepsilon$.
If $4r<\lceil m^{\varepsilon/(1-\varepsilon)}\rceil$, then $F$ is $2$-degenerate, and thus it has density at most $2$.
Hence, we can assume $r=\Omega(m^{\varepsilon/(1-\varepsilon)})$.  Let $M$ be the set of vertices of $G^\varepsilon$ of degree three, for each vertex $v\in V(F)$ let $B_v$ be the vertex set of the
subgraph of $G^\varepsilon$ contracted to $v$, and let $v$ be the vertex of $F$ with $|B_v\cap M|$ minimum.  Note that $\deg_F(v)\le 2+|B_v\cap M|$.
Furthermore, since the sets $B_u$ for $u\in V(F)$ are pairwise disjoint, we have
$(\deg_F(v)+1)|B_v\cap M|\le |B_v\cap M|+\sum_{u\in N(v)} |B_u\cap M|\le |M|=m$.
Consequently, $\deg_F(v)\le 2+\min(|B_v\cap M|, m/|B_v\cap M|)\le 2+\sqrt{m}=O(r^{\frac{1}{2\varepsilon}-\frac{1}{2}})$.
Therefore, $\nabla_{\GG_3^\varepsilon}(r)=O(r^{\frac{1}{2\varepsilon}-\frac{1}{2}})$.
\end{proof}
Furthermore, note that if $G$ is an expander, then $G$ contains as an $O(\log m)$-shallow minor a clique with $\Omega(\sqrt{m/\log m})$ vertices by Theorem~\ref{thm-prs},
and thus if $m=\Theta(r^{\frac{1}{\varepsilon}-1})$, we conclude $G^\varepsilon$ contains as an $O(r\log r)$-shallow minor a clique with $\Omega(r^{\frac{1}{2\varepsilon}-\frac{1}{2}}/\pll r)$ vertices.
Consequently, $\nabla_{\GG_3^\varepsilon}(r)=\Omega(r^{\frac{1}{2\varepsilon}-\frac{1}{2}}/\pll r)$; hence, the analysis of this example cannot be substantially improved.

This construction also gives a lower bound for $B_\varepsilon$ that matches the upper bound from Corollary~\ref{cor-known}.
\begin{lemma}\label{lemma-scl}
For $0<\varepsilon\le 1$, $B_\varepsilon\ge \frac{1}{\varepsilon}-1$.
\end{lemma}
\begin{proof}
Since $B_1=0$ by Corollary~\ref{cor-known}, we can assume $\varepsilon<1$.
Let $\GG_a$ be the class of all graphs.  By Lemma~\ref{lemma-subdivsep}, we have $s_{\GG_a^\varepsilon}(n)=\Theta(n^{1-\varepsilon})$.
For a sufficiently large integer $r$, let $m=\lfloor r^{\frac{1}{\varepsilon}-1}\rfloor$.  The graph $K_m^\varepsilon$ contains the clique $K_m$
as an $r$-shallow minor, implying $\nabla_{\GG_a^\varepsilon}(r)=\Omega(r^{\frac{1}{\varepsilon}-1})$.
\end{proof}

Finally, a similar idea enables us to obtain a better bound for $b_\varepsilon$ in the range $\tfrac{1}{2}\le \varepsilon\le 1$.
\begin{lemma}\label{lemma-sgr}
For $\tfrac{1}{2}\le \varepsilon\le 1$, $b_\varepsilon=0$.
\end{lemma}
\begin{proof}
Since $b_1=0$ by Corollary~\ref{cor-known}, we can assume $\varepsilon<1$.
For a graph $G$ with $m$ vertices, let $G'$ denote the graph obtained from $G$ by subdividing each edge $\bigl\lceil m^{\frac{2\varepsilon-1}{2-2\varepsilon}}\bigr\rceil$ times.
Let $\GG$ consist of all induced subgraphs of the graphs $G'$ for all planar graphs $G$.  All graphs in $\GG$ are planar, and thus $\nabla_\GG(r)\le 3=O(r^0)$ holds for every $r\ge 0$.
Standard isoperimetric inequalities applied with $G$ being a $(t\times t)$-grid for $t=\Theta(n^{1-\varepsilon})$ (so that $|V(G')|=\Theta(n)$)
show that every balanced separator in $G'$ has size $\Omega(t)=\Omega(n^{1-\varepsilon})$, implying $s_{\GG'}(n)=\Omega(n^{1-\varepsilon})$.
Conversely, Lemma~\ref{lemma-twtos} applied with $f(n)=O(\sqrt{n})$ and $t(m)=\bigl\lceil m^{\frac{2\varepsilon-1}{2-2\varepsilon}}\bigr\rceil$ (so that
$p(n)=\Theta(n^{2-2\varepsilon})$) implies $s_{\GG'}(n)=O(n^{1-\varepsilon})$.
\end{proof}

Let us summarize our findings:
We have $b_\varepsilon=0$ when $\tfrac{1}{2}\le \varepsilon\le 1$,
$$\frac{1}{2\varepsilon}-1\le b_\varepsilon\le \frac{1}{2\varepsilon}-\frac{1}{2}$$
when $0<\varepsilon<\tfrac{1}{2}$,
and $$B_\varepsilon=\frac{1}{\varepsilon}-1$$
when $0<\varepsilon\le 1$.
In particular, if $\varepsilon<1$, then $b_\varepsilon\neq B_\varepsilon$.

The bounds for $b_\varepsilon$ differ by at most $1/2$.
It is unclear whether the upper or the lower bound can be improved.  While the fact that $b_{1/2}=0$ matches the lower bound
suggests that a better construction improving the upper bound in general could exist, it is also plausible that this is just
a ``dimension 2'' artifact and in fact the lower bound might be possible to improve for $\varepsilon<1/2$ (possibly leading to discontinuity
of $b_\varepsilon$ at $\varepsilon=1/2$).

Instead of $b_\varepsilon$ and $B_\varepsilon$, the following parameters (constraining $s_\GG(n)$ from below
by $\Omega(n^{1-\varepsilon})$ and from above by $O(n^{1-\varepsilon})$, rather than by $\Theta(n^{1-\varepsilon})$) might be considered
more natural. For $0<\varepsilon\le 1$,
\begin{itemize}
\item let $b'_\varepsilon$ denote the supremum of real numbers $b'$ for which every hereditary class $\GG$ of graphs such that $s_\GG(n)=\Omega(n^{1-\varepsilon})$
satisfies $\nabla_\GG(r)=\Omega(r^{b'})$, and
\item let $B'_\varepsilon$ denote the infimum of real numbers $B'$ for which every hereditary class $\GG$ of graphs such that $s_\GG(n)=O(n^{1-\varepsilon})$
satisfies $\nabla_\GG(r)=O(r^{B'})$.
\end{itemize}
By Corollary~\ref{cor-lower} and Theorem~\ref{thm-upper}, we have
$$\max\Bigl(\frac{1}{2\varepsilon}-1,0\Bigr)\le b'_\varepsilon\le b_\varepsilon\le B_\varepsilon\le B'_\varepsilon\le \frac{1}{\varepsilon}-1.$$
Hence, $b'_\varepsilon=b_\varepsilon=0$ when $\tfrac{1}{2}\le \varepsilon\le 1$,
$$\frac{1}{2\varepsilon}-1\le b'_\varepsilon\le b_\varepsilon\le \frac{1}{2\varepsilon}-\frac{1}{2}$$
when $0<\varepsilon<\tfrac{1}{2}$,
and $$B'_\varepsilon=B_\varepsilon=\frac{1}{\varepsilon}-1$$
when $0<\varepsilon\le 1$.  It seems likely that $b'_\varepsilon=b_\varepsilon$ when $0<\varepsilon<\tfrac{1}{2}$ as well,
but this is not obvious: Consider any hereditary class $\GG'$ such that $s_{\GG'}(n)=\Omega(n^{1-\varepsilon})$.
Without loss of generality, we can assume $\GG'$ is monotone (closed under subgraphs) rather than just hereditary.
To try to transform $\GG'$ to a hereditary class $\GG$ with $s_\GG(n)=\Theta(n^{1-\varepsilon})$, it is natural
to let $\GG$ consist of all graphs $G\in \GG$ such that every induced subgraph $H$ of $G$ satisfies $s(H)=O(|V(H)|^{1-\varepsilon})$.
However, as we have to ensure this bound holds for all induced subgraphs $H$ (in order for $\GG$ to be hereditary),
it is not clear the remaining graphs are sufficient to enforce $s_\GG(n)=\Omega(n^{1-\varepsilon})$.

\bibliographystyle{siam}
\bibliography{../data.bib}

\end{document}